\documentclass[a4paper,12pt, leqno]{article}

\usepackage{biblatex}
\usepackage[utf8]{inputenc}
\usepackage{graphicx}
\usepackage{tikz}
\tikzset{
    use page relative coordinates/.style={
        shift={(current page.south west)},
        x={(current page.south east)},
        y={(current page.north west)}
    },
}
\usetikzlibrary{decorations.pathreplacing,angles,quotes}
\usepackage{mathtools}
\usepackage{esint}
\usepackage{amsthm,amstext,amsfonts,bm,amssymb}
\usepackage{hyperref}
\usepackage{cleveref}
\usepackage{color}   
\AtBeginDocument{\hypersetup{pdfborder={0 0 1}}}
\usepackage{amsmath}
\usepackage{amssymb}
\usepackage{amsthm}
\theoremstyle{plain}
\usepackage{etoolbox}
\usepackage{bigints}
\usepackage{physics}
\usepackage{comment}
\usepackage{thm-restate}
\usepackage{thmtools}

\usepackage{authblk}
\usepackage{arydshln}
\usepackage{enumerate}
\hypersetup{
linktocpage= true
}

\newtheorem{theorem}{Theorem}[section]
\newtheorem{lemma}[theorem]{Lemma}

\newtheorem{proposition}[theorem]{Proposition}

\newtheorem{remark}[theorem]{Remark}
\numberwithin{equation}{section}



\title{On Counterexamples to Interior $C^2$ Estimates for Monge-Amp\`ere Type Equations}
\author{Cheuk Yan Fung}
\date{Sept 2025}

\begin{document}
\maketitle
\begin{abstract}
We modify Pogorelov’s classic construction to demonstrate the absence of a priori $C^2$ estimates for the equations $\det(D^2 u \pm Du \otimes Du) = f(x)$ in dimension $n \ge 3$. We construct a sequence of solutions $z_\varepsilon$ with second derivatives blowing up at the origin as $\varepsilon \rightarrow 0$, while the corresponding right-hand sides $f_\varepsilon$ admit uniform $C^2$ estimates. Specifically, the counterexamples are given by $z_\varepsilon(x_1, \dots, x_n) = (1+x_1^2)(1+x_2^2)(\varepsilon^2 + \eta^2)^{\alpha/2},$
where $\eta = \sqrt{x_3^2 + \dots + x_n^2}$ and $\alpha = 2 - \frac{2}{n}$.
\end{abstract}
\section{Introduction}
In dimension $n=2$, Heinz \cite{Heinz1959} established interior $C^2$ estimates for the standard Monge-Amp\`ere equation 
\[\det(D^2 u) = f.\] 
However, Heinz and Lewy (see \cite{Schulz1990}) showed that such estimates can fail for perturbed equations. Specifically, they constructed a sequence of solutions for the equation
\begin{equation*}
    \det(D^2 u + c_{ij}(Du)) = f(x)
\end{equation*}
that lack a uniform $C^2$ bound, demonstrating that interior $C^2$ estimates fail without suitable structural conditions on the coefficient $c_{ij}$. In dimensions $n\geq 3$, Pogorelov \cite{Pogorelov1978Minkowski} constructed a convex $C^{1,1-\frac{2}{n}}$ solution to the Monge-Amp\`ere equation with $f \in C^2$, demonstrating that the interior $C^2$ estimate fails for this equation.

Following Pogorelov’s seminal construction, subsequent work has extensively investigated more general fully nonlinear equations of the form
\begin{equation*}
F(D^2u + a(x) Du \otimes Du + b(x)|Du|^2I + c_{ij}(x)) = f(x,u,Du),
\end{equation*}
where $F$ is assumed to be concave, elliptic, and positive, and $a, b, c_{ij}$ are $C^2$ functions. There has been significant interest in the geometric counterpart of this problem, namely equations of the form $F(g^{-1}W) = f$. Here, $W$ is the augmented Hessian tensor
\begin{equation*}
W = \nabla^2u + a(x) du \otimes du + b(x)|\nabla u|^2 g + c(x),
\end{equation*}
where the Hessian and gradient norm are defined with respect to a Riemannian metric $g$.

The sign of the coefficient $b(x)$ plays a crucial role in establishing a priori Hessian estimates. For the case $b(x) < 0$, Chen \cite{Chen2005LocalEstimates} derived local estimates for a broad class of equations. Under certain structural conditions on the function $f$, Chen utilized the maximum principle to obtain the Hessian estimate assuming $b(x) < -\delta_1$ and $a(x) + nb(x) < -\delta_2$. In particular, Chen's result applies to the Schouten tensor equation
\begin{equation*}
\sigma_k^{1/k}(g^{-1}(\nabla^2 u + du \otimes du - \tfrac{1}{2}|\nabla u|^2 g + A_g)) = f(x) e^{-2u}.
\end{equation*}
The study of the Schouten tensor equation was initiated by Viaclovsky \cite{Via00a}. In \cite{Via02}, he derived global $C^2$ estimates for solutions on compact manifolds, with bounds depending on $C^0$ estimates. Local $C^2$ estimates were obtained for this equation by Chang, Gursky, and Yang \cite{Chang2002} in the case $k=2, n=4$, and by Guan and Wang \cite{GuanWang2003} for all $k \leq n$. Subsequently, Gursky and Viaclovsky \cite{GV2007} established refined local estimates with explicit dependence on the domain radius. Guan and Wang \cite{GuanWang2004} established local $C^2$ estimates for the quotient equations. For locally conformally flat manifolds, Guan and Wang \cite{GuanWang2003Crelle} and Li and Li \cite{LiLi2003}
established the existence of a metric $g$ such that its $\sigma_k$-curvature is a positive constant. 

Conversely, for the case $b(x) > 0$, Sheng, Trudinger, and Wang \cite{ShengTrudingerWang2007} modified the Heinz-Lewy counterexample for the equation
\begin{equation*}
\det(D^2 u + |Du|^2 I + c_{ij}(x)) = f,
\end{equation*}
where $c_{ij}$ and $f$ are $C^2$ functions. 
They showed that this equation admits no interior a priori $C^{1,1}$ estimates for solutions whose eigenvalues lie in the positive cone $\Gamma_n$.

A natural question is whether the interior $C^2$ estimate can be established in the case where $b(x) \equiv 0$, while $a(x)$ remains non-zero. We answer this question in the negative. Specifically, we modify Pogorelov’s classic construction to demonstrate the absence of a priori $C^2$ estimates for the equations:
\begin{equation}
\label{duduequation}
\det(D^2 u \pm Du \otimes Du) = f(x).
\end{equation}
\begin{theorem}
Let $n\geq 3$. The interior $C^2$ estimate fails for the equations (\ref{duduequation}). Specifically, there exist constants $\rho>0$ and $\varepsilon_0 > 0$ such that the sequence of smooth functions $\{z_\varepsilon\}$ defined on $B_\rho(0) \subset \mathbb{R}^n$ satisfies the following for all $\varepsilon \in (0, \varepsilon_0)$:
\begin{enumerate} [(a)]
\item As $\varepsilon \rightarrow 0$, $|D^2z_{\varepsilon}(0)| \rightarrow \infty.$
\item The sequence $\{z_\varepsilon\}$ admits uniform $C^{1, 1-\frac{2}{n}}$ estimates in $\varepsilon$ on $B_\rho(0)$.
\item The functions $f_\varepsilon := \det((z_\varepsilon)_{ij} \pm (z_\varepsilon)_i (z_\varepsilon)_j)$ admit uniform $C^2$ estimates in $\varepsilon$ on $B_\rho(0)$.
\item The augmented Hessian matrix $ D^2 z_\varepsilon \pm Dz_\varepsilon \otimes Dz_\varepsilon$ is positive definite on $B_\rho(0)$.
\end{enumerate}
Specifically, the functions are given by
\begin{equation*}
z_\varepsilon(x_1, \dots, x_n) = (1+x_1^2)(1+x_2^2)(\varepsilon^2 + \eta^2)^{\alpha/2},
\end{equation*}
where $\eta = \sqrt{x_3^2 + \dots + x_n^2}$ and $\alpha = 2 - \frac{2}{n}$.
\end{theorem}

We will use the notations $u_\varepsilon(x_1,x_2,\eta) = z_\varepsilon(x_1, \dots, x_n)$ and $r_\varepsilon= \sqrt{\eta^2+\varepsilon^2}$ in the article. For simplicity, we write $z$, $u$, and $r$ in place of $z_\varepsilon$, $u_\varepsilon$, and $r_\varepsilon$. The verification of (a) follows from direct calculation:
\begin{equation*}
z_{33} = \alpha (\alpha-2) (1+x_1^2)(1+x_2^2)(\varepsilon^2 + \eta^2)^{\frac{\alpha-4}{2}}x_3^2 + \alpha(1+x_1^2)(1+x_2^2)(\varepsilon^2 + \eta^2)^{\frac{\alpha-2}{2}}.
\end{equation*}
Evaluating at the origin, we find:
\begin{equation*}
z_{33}(0) = \alpha \varepsilon^{\alpha-2} \to \infty \text{ as } \varepsilon \to 0,
\end{equation*}
since $\alpha < 2$.
It is immediate to see why (b) holds by considering $\varepsilon$ as an additional variable and using the well known fact that $|x|^{2-2/n}$ is $C^{1,1-2/n}$.
\begin{remark}
The simpler Pogorelov-type function $z(x_1, \dots, x_n) = (1+x_1^2)\eta^\alpha$ (where $\eta = \sqrt{x_2^2 + \dots + x_n^2}$) fails to provide a valid counterexample. Direct computation of the augmented determinant yields an expansion of the form:
\begin{equation*}
\det(z_{ij} + z_i z_j) = g_1 \eta^{n\alpha - 2(n-1)} + g_2 \eta^{(n+1)\alpha - 2(n-1)},
\end{equation*}
where $g_1$ and $g_2$ are $C^2$ functions. For example, in dimension $n=3$, this simplifies to:
\begin{equation*}
    \det(z_{ij} + z_i z_j) = g_1 \eta^{3\alpha - 4} + g_2 \eta^{4\alpha - 4}.
\end{equation*}
For the determinant $f$ to be of class $C^2$ at $\eta=0$, the powers of $\eta$ must be at least $2$ or exactly $0$. This leads to two cases:
\begin{enumerate}
\item Case 1: $3\alpha - 4 \geq 2$ and $4\alpha - 4 \geq 2$. The first inequality implies $\alpha \geq 2$, which yields a smooth solution $z$ and therefore does not provide a counterexample.
\item Case 2: $3\alpha - 4 = 0$ and $4\alpha - 4 \geq 2$. The first equality implies $\alpha = 4/3$. However, substituting this into the second term yields an exponent of $4(4/3) - 4 = 4/3$, which is less than $2$.
\end{enumerate}
Hence, it is necessary to use the modified Pogorelov-type function $z = (1+x_1^2)(1+x_2^2)\eta^\alpha$. The additional factor $(1+x_2^2)$ increases the exponents of both terms in the determinant expansion, allowing a choice of $\alpha < 2$ while maintaining $C^2$ regularity of $f$. For example, in dimension $3$, the exponents increase from $3\alpha-4$ and $4\alpha-4$ to $3\alpha-2$ and $4\alpha-2$, respectively. The parameter $\varepsilon$ is then introduced to prevent the determinant from degenerating to zero at the origin.
\end{remark}

Outline: The remainder of this paper is organized as follows. In Section \ref{step1}, we establish a reduction formula for the determinant of the augmented Hessian. In Section \ref{step2}, we prove the uniform $C^2$ regularity of the determinant. Finally, Section \ref{step3} verifies the admissibility of the constructed sequence.
\section{Reduction of the Determinant}
\label{step1}
In this section, we consider the following choice of $z$:
\begin{equation*}
z(x_1,x_2,x_3,\dots,x_n) = u(x_1,x_2,\sqrt{x_3^2+x_4^2+\dots+x_n^2}).
\end{equation*}
By applying Schur’s formula, the matrix determinant lemma, and the Sherman-Morrison formula, we express the determinant of $z_{ij} \pm z_iz_j$ in terms of $u_{ij}$ and $u_i$. We begin by recalling these results.
\begin{lemma}[Schur's formula]
\label{lemma2_1}
Let $M$ be a block matrix of the form
\begin{equation*}
M = \begin{bmatrix} A & B \\ C & D \end{bmatrix}
\end{equation*}
and $D$ be invertible, then
\begin{equation*}
\det(M) = \det(D) \det(A - BD^{-1}C).
\end{equation*}
\end{lemma}
\begin{lemma}[Matrix determinant lemma]
\label{lemma2_2}
Let $A$ be an invertible square matrix and let $a, b$ be column vectors. Then
\begin{equation*}
\det(A + ab^T) = (1 + b^T A^{-1} a) \det A.
\end{equation*}
\end{lemma}

\begin{lemma}
\label{lemma2_3}
(Sherman-Morrison formula) Suppose that $A \in \mathbb{R}^{n \times n}$ is an invertible square matrix and $a,b \in \mathbb{R}^n$ are column vectors. Then $A+a b^T$ is invertible if and only if $1+b^TA^{-1}a\neq 0$. In this case,
\begin{equation*}
    (A+ab^T)^{-1} = A^{-1} - \frac{A^{-1}ab^T A^{-1}}{1+b^T A^{-1} a }.
\end{equation*}
\end{lemma}
We now apply these formulas to compute the determinant of $z_{ij} \pm z_iz_j$ for points where $\eta \neq 0$. We remark that at the points where $\eta = 0$, direct computation shows that the augmented Hessian becomes block diagonal:
\begin{equation}
\label{eta=0 augmented}
    (z_{ij} \pm z_i z_j)\big|_{\eta=0} = 
\begin{bmatrix}
z_{11} \pm z_1^2 & z_{12} \pm z_1 z_2 & 0 & \dots & 0 \\
z_{12} \pm z_1 z_2 & z_{22} \pm z_2^2 & 0 & \dots & 0 \\
0 & 0 & \multicolumn{3}{c}{} \\
\vdots & \vdots & \multicolumn{3}{c}{\displaystyle \Biggl[ \alpha(1+x_1^2)(1+x_2^2)\varepsilon^{\alpha-2} \delta_{ij} \Biggr]} \\
0 & 0 & \multicolumn{3}{c}{}
\end{bmatrix}.
\end{equation}

\begin{lemma}
\label{det_formula}
Let $z(x_1, \dots, x_n) = u(x_1, x_2, \eta)$ with $\eta = \sqrt{x_3^2 + \dots + x_n^2}$. Suppose that $\eta\neq 0$. Then the determinant of the augmented Hessian satisfies:
\begin{equation}
\label{det_formula_1}
\begin{split}
    \det(z_{ij} \pm z_i z_j) = & \left(\frac{u_3}{\eta}\right)^{n-3}  \det(u_{ij} \pm u_i u_j).
\end{split}
\end{equation}
Here, the indices $1, 2, 3$ for $u$ denote derivatives with respect to $x_1, x_2$, and $\eta$ respectively.
\end{lemma}

\begin{proof}
Direct computation shows:
\begin{equation}
\label{fullz_ijz_iz_j}
\begin{split}
&(z_{ij} \pm z_i z_j)\\
=&           
\begin{bmatrix}
u_{11} \pm u_1^2 & u_{12} \pm u_1 u_2 & \dots & (u_{13} \pm u_1 u_3) \frac{x_n}{\eta} \\
u_{12} \pm u_1 u_2 & u_{22} \pm u_2^2 & \dots & (u_{23} \pm u_2 u_3) \frac{x_n}{\eta} \\
(u_{13} \pm u_1 u_3)\frac{x_3}{\eta} & (u_{23} \pm u_2 u_3)\frac{x_3}{\eta} & \multicolumn{2}{c}{} \\
\vdots & \vdots & \multicolumn{2}{c}{\displaystyle \Biggl[ \frac{u_3}{\eta}\delta_{ij}+\left(\frac{u_{33} \pm u_3^2}{\eta^2}-\frac{u_3}{\eta^3}\right)x_i x_j \Biggr]} \\
(u_{13} \pm u_1 u_3)\frac{x_n}{\eta} & (u_{23} \pm u_2 u_3)\frac{x_n}{\eta} & \multicolumn{2}{c}{}
\end{bmatrix}
\end{split}.
\end{equation}
Define the $(n-2) \times (n-2)$ block $M$ by
\begin{equation}   
\label{M}
M_{ij} := \frac{u_3}{\eta}\delta_{ij} + \left( \frac{u_{33} \pm u_3^2}{\eta^2} - \frac{u_3}{\eta^3} \right) x_i x_j, \quad i,j \in \{3, \dots, n\}. 
\end{equation}

We may then write the augmented Hessian $z_{ij} \pm z_i z_j$ in block form as
\begin{equation*}
\begin{bmatrix}
A & C^T \\C & M
\end{bmatrix},
\end{equation*}
where $A$ is a $2 \times 2$ matrix and $C$ is a $(n-2) \times 2$ matrix. By Schur's formula (Lemma \ref{lemma2_1}), we have
\begin{equation*}
\det (z_{ij} \pm z_i z_j) = \det(M) \det(A - C^T M^{-1} C).
\end{equation*}
To compute $\det(M)$, we apply the matrix determinant lemma (Lemma \ref{lemma2_2}) to obtain:
\begin{equation}
\label{detM}
\begin{split}
\det(M) &= \det\left( \frac{u_3}{\eta} I \right) \left[ 1 + \sum_{j=3}^n \left( \frac{u_{33} \pm u_3^2}{\eta^2} - \frac{u_3}{\eta^3} \right) x_j^2 \cdot \frac{\eta}{u_3} \right] \\
&= \left( \frac{u_3}{\eta} \right)^{n-2} \left[ 1 + \left( \frac{u_{33} \pm u_3^2}{\eta^2} - \frac{u_3}{\eta^3} \right) \frac{\eta^3}{u_3} \right] \\
&= \left( \frac{u_3}{\eta} \right)^{n-2} \left[ 1 + \frac{\eta (u_{33} \pm u_3^2)}{u_3} - 1 \right] \\
&= \left( \frac{u_3}{\eta} \right)^{n-3} (u_{33} \pm u_3^2).
\end{split}
\end{equation}
By the Sherman-Morrison formula (Lemma \ref{lemma2_3}), we have:
\begin{align*}
M^{-1}_{ij} &= \left( \frac{u_3}{\eta} I \right)^{-1} - \frac{\left( \frac{u_3}{\eta} I \right)^{-1} \left( \frac{u_{33} \pm u_3^2}{\eta^2}-\frac{u_3}{\eta^3} \right) (x_i x_j) \left( \frac{u_3}{\eta} I \right)^{-1}}{1 + \sum_{k=3}^n x_k^2 \left( \frac{u_{33} \pm u_3^2}{\eta^2}-\frac{u_3}{\eta^3} \right) \frac{\eta}{u_3}} \\
&= \frac{\eta}{u_3} I - \left( \frac{\eta^2}{u_3^2} \right) \left( \frac{u_{33} \pm u_3^2}{\eta^2}-\frac{u_3}{\eta^3} \right) (x_i x_j) \left[ 1 + \eta^2 \left( \frac{u_{33} \pm u_3^2}{\eta^2}-\frac{u_3}{\eta^3} \right) \frac{\eta}{u_3} \right]^{-1} \\
&= \frac{\eta}{u_3} I - \left( \frac{u_{33} \pm u_3^2}{u_3^2} - \frac{1}{\eta u_3} \right) (x_i x_j) \left( \frac{\eta (u_{33} \pm u_3^2)}{u_3} \right)^{-1} \\
&= \frac{\eta}{u_3} I - \left( \frac{1}{\eta u_3} - \frac{1}{\eta^2 (u_{33} \pm u_3^2)} \right) x_i x_j.
\end{align*}
Direct computation gives:
\begin{align*}
    (C^T M^{-1} C)_{11} &= \sum_{k,l=3}^n \frac{(u_{13} \pm u_1 u_3)^2 x_k x_l}{\eta^2} \left[ \frac{\eta}{u_3}\delta_{kl} - \left(\frac{1}{u_3 \eta} - \frac{1}{\eta^2(u_{33} \pm u_3^2)}\right)x_k x_l \right] \\
    &= \frac{(u_{13} \pm u_1 u_3)^2}{\eta^2} \left[ \frac{\eta^3}{u_3} - \left(\frac{\eta^3}{u_3} - \frac{\eta^2}{u_{33} \pm u_3^2}\right) \right] \\
    &= \frac{(u_{13} \pm u_1 u_3)^2}{u_{33} \pm u_3^2}.
\end{align*}
Similarly, for $i, j \in \{1, 2\}$, we have:
\begin{equation*}
    (C^T M^{-1}C)_{ij} = \frac{(u_{i3} \pm u_i u_3)(u_{j3} \pm u_j u_3)}{u_{33} \pm u_3^2}.
\end{equation*}
Substituting this back into the Schur formula, we obtain:
\begin{align*}
&\det(z_{ij} \pm z_i z_j) \\
&= \left(\frac{u_3}{\eta}\right)^{n-3} (u_{33} \pm u_3^2) \cdot \det \left( A_{ij} - \frac{(u_{i3} \pm u_i u_3)(u_{j3} \pm u_j u_3)}{u_{33} \pm u_3^2} \right)_{i,j \in \{1,2\}} \\
&= \left(\frac{u_3}{\eta}\right)^{n-3} \frac{\det \left( (u_{ij} \pm u_i u_j)(u_{33} \pm u_3^2) - (u_{i3} \pm u_i u_3)(u_{j3} \pm u_j u_3) \right)_{1 \le i,j \le 2} }{u_{33} \pm u_3^2} \\
&= \left(\frac{u_3}{\eta}\right)^{n-3} \det(u_{ij} \pm u_i u_j),
\end{align*}
where the last equality follows by direct computation.
\end{proof}
\section{Uniform $C^2$ Regularity of the Determinant}
\label{step2}
With the choice of the function $u$, we now show that the determinant of the augmented Hessian remains bounded in the $C^2$ norm on a fixed neighborhood of the origin, with estimates independent of $\varepsilon$ as $\varepsilon \rightarrow 0$.
\begin{proposition} 
Let $z(x_1,\dots,x_n) = u(x_1,x_2,\eta) = (1+x_1^2)(1+x_2^2)(\varepsilon^2+ \eta^2)^{\frac{\alpha}{2}}$ and $\alpha = 2-\frac{2}{n}$. Then for any fixed radius $\rho > 0$, the functions $
\det(z_{ij} \pm z_i z_j)$
admit uniform $C^2$ estimates on $B_\rho(0)$ as $\varepsilon \rightarrow 0$. 
\end{proposition}
\begin{proof}
We first consider the case $\eta \neq 0$.
By Lemma \ref{det_formula}, the determinant is given by (\ref{det_formula_1}) for $\eta \neq 0$. Direct computation shows:
\begin{equation}
\label{fullu}
\begin{split}
    u_1 &= z_1 = 2x_1(1+x_2^2)(\varepsilon^2+\eta^2)^{\frac{\alpha}{2}} \\
    u_2 &= z_2 = 2x_2(1+x_1^2)(\varepsilon^2+\eta^2)^{\frac{\alpha}{2}} \\
    u_3  &= \alpha(1+x_1^2)(1+x_2^2)\eta(\varepsilon^2+\eta^2)^{\frac{\alpha}{2}-1} \\
    u_{11} &= z_{11} = 2(1+x_2^2)(\varepsilon^2+\eta^2)^{\frac{\alpha}{2}} \\ 
        u_{12} &= z_{12} = 4x_1x_2(\varepsilon^2+\eta^2)^{\frac{\alpha}{2}} \\
    u_{22} &= z_{22} = 2(1+x_1^2)(\varepsilon^2+\eta^2)^{\frac{\alpha}{2}} \\
    u_{31}  &= 2\alpha x_1(1+x_2^2)\eta(\varepsilon^2+\eta^2)^{\frac{\alpha}{2}-1} \\
    u_{32}  &= 2\alpha x_2(1+x_1^2)\eta(\varepsilon^2+\eta^2)^{\frac{\alpha}{2}-1} \\
    u_{33} &= \alpha(\alpha-2)(1+x_1^2)(1+x_2^2)\eta^2 (\varepsilon^2+\eta^2)^{\frac{\alpha}{2}-2}\\
    &+ \alpha(1+x_1^2)(1+x_2^2)(\varepsilon^2+\eta^2)^{\frac{\alpha}{2}-1}.
\end{split}
\end{equation}
Using these formulas, the columns of the augmented Hessian matrix are given by:
\begin{equation}
\label{1col}
    u_{m1} \pm u_m u_1 = \begin{bmatrix}
          2(1+x_2^2)r^{\alpha} \pm 4x_1^2(1+x_2^2)^2 r^{2\alpha} \\
          4x_1x_2r^{\alpha} \pm 4x_1x_2(1+x_1^2)(1+x_2^2)r^{2\alpha}\\
    2\alpha x_1(1+x_2^2)\eta r^{\alpha-2} \pm 2\alpha x_1 (1+x_1^2)(1+x_2^2)^2\eta r^{2\alpha -2} \\
           
    \end{bmatrix}
\end{equation}
\begin{equation}
\label{2col}
    u_{m2} \pm u_m u_2 = \begin{bmatrix}
    4x_1x_2r^{\alpha} \pm 4x_1x_2(1+x_1^2)(1+x_2^2)r^{2\alpha}\\
     2(1+x_1^2)r^\alpha \pm 4x_2^2(1+x_1^2)^2 r^{2\alpha} \\
    2\alpha x_2(1+x_1^2)\eta r^{\alpha-2} \pm 2\alpha x_2 (1+x_2^2)(1+x_1^2)^2\eta r^{2\alpha -2} \\
           
    \end{bmatrix}
\end{equation}
\begin{equation}
\label{3col}
    u_{m3} \pm u_m u_3 = \begin{bmatrix}
    2\alpha x_1(1+x_2^2)\eta r^{\alpha-2} \pm 2\alpha x_1 (1+x_1^2)(1+x_2^2)^2\eta r^{2\alpha -2} \\
    2\alpha x_2 (1+x_1^2)\eta r^{\alpha-2} \pm 2\alpha x_2 (1+x_2^2)(1+x_1^2)^2 \eta r^{2\alpha-2} \\
    \alpha(\alpha-2)(1+x_1^2)(1+x_2^2) r^{\alpha-4}\eta^2 + \alpha(1+x_1^2)(1+x_2^2)r^{\alpha-2} \\
    \pm \alpha^2(1+x_1^2)^2(1+x_2^2)^2 \eta^2 r^{2\alpha-4}
    \end{bmatrix}.
\end{equation}

The leading factor simplifies to:
\begin{equation*}
    \left(\frac{u_3}{\eta}\right)^{n-3} = \alpha^{n-3}(1+x_1^2)^{n-3}(1+x_2^2)^{n-3}r^{(\alpha-2)(n-3)}.
\end{equation*}
To analyze the uniform regularity of $\det(u_{ij} \pm u_i u_j)$, we perform scaling operations on the matrix $(u_{ij} \pm u_i u_j)$, whose columns are given by $(\ref{1col})$, $(\ref{2col})$, and $(\ref{3col})$ respectively. We factor $r^{\alpha}$ from the first and second rows, $\eta r^{\alpha-2}$ from the third row, and $\frac{\eta}{r^2}$ from the third column. This yields:
\begin{equation*}\left(\frac{u_3}{\eta}\right)^{n-3} \det(u_{ij} \pm u_i u_j) = \eta^2 r^{n\alpha-2n+2} \det \mathcal{M}.
\end{equation*}
The entries of the matrix $\mathcal{M} = (m_{ij})$ are given by:
\begin{equation}
\label{anotherM}
    \mathcal{M} = \begin{bmatrix}
        g_{11} \pm h_{11} r^\alpha & g_{12} \pm h_{12} r^\alpha  &  g_{13} \pm h_{13} r^{\alpha} \\
        g_{21} \pm h_{21} r^\alpha & g_{22} \pm h_{22} r^\alpha  &  g_{23} \pm h_{23} r^{\alpha} \\
        g_{31} \pm h_{31} r^{\alpha} & g_{32} \pm h_{32} r^{\alpha} & g_{33} \pm h_{33}  r^{\alpha} + k_{33}\eta^{-2}r^2
    \end{bmatrix},
\end{equation}
where $g_{ij},h_{ij},k_{ij}$ are $C^2$ functions with uniform $C^2$ estimates as $\varepsilon \rightarrow 0$.

Direct computation shows
\begin{equation*}
\left(\frac{u_3}{\eta}\right)^{n-3} \det(u_{ij} \pm u_i u_j) = r^{n\alpha-2n+2} \eta^2 g_1 + r^{(n+1)\alpha-2n+2} \eta^2 g_2 + r^{n\alpha-2n+4} g_3 +g_4,
\end{equation*}
where $g_1,g_2,g_3,g_4$ are $C^2$ functions with uniform $C^2$ estimates as $\varepsilon \to 0$.

To establish uniform $C^2$ estimates, it suffices to verify that the powers of $r$ are either at least $2$ or exactly $0$. By choosing $\alpha = 2 - \frac{2}{n}$, we have $n\alpha - 2n + 2 = 0$, $n\alpha-2n+4= 2$ and $(n+1)\alpha - 2n + 2 = 2-\frac{2}{n}.$ Consequently, the terms
$ r^{n\alpha-2n+2} \eta^2 g_1$ and $r^{n\alpha-2n+4} g_3$ admit a uniform $C^2$ estimate. To show the uniform $C^2$ regularity for the term $ r^{(n+1)\alpha-2n+2} \eta^2g_2$, it suffices to consider the case where the derivatives act on the $r$ factor for $k,l\geq 3$:

\begin{align*}
    \left|\frac{\partial}{\partial x_k} \left( r^{2-\frac{2}{n}} \right) \eta^2\right| &= \left|C x_k r^{-\frac{2}{n}} \eta^2 \right|\leq Cr^{3-\frac{2}{n}}\\
    \left|\frac{\partial^2}{\partial x_k \partial x_l} \left( r^{2-\frac{2}{n}} \right) \eta^2\right| &= \left|C \delta_{kl} r^{-\frac{2}{n}} \eta^2 + C x_k x_l r^{-2-\frac{2}{n}} \eta^2\right| \\
    &\leq C r^{2-\frac{2}{n}} \leq C,
\end{align*}
where we used $\eta \leq r$ and $|x_k|, |x_l| \leq r$. 

For the case $\eta =0 $, by the definition of $z(x_1,\dots,x_n) =  (1+x_1^2)(1+x_2^2)(\varepsilon^2+ \eta^2)^{\frac{\alpha}{2}} $, we can see that $z$ is smooth everywhere, including at the points where $\eta = 0$. Therefore, it suffices to establish the uniform estimate $C^2$ at the points where $\eta \neq 0$, since the estimate at $\eta = 0$ follows by continuity.
\end{proof}
\section{Admissibility of the Example}
\label{step3}
In this final section, we verify the admissibility of our constructed function. A function is said to be admissible if its augmented Hessian is positive definite. We will use Sylvester's criterion to establish the positive definiteness of the augmented Hessian.
\begin{lemma}
    (Sylvester's criterion) An $n \times n$ symmetric matrix $M$ is positive definite if and only if all leading principal submatrices of $M$ have a positive determinant.
\end{lemma}
It is equivalent to checking that the trailing principal submatrices have positive determinants by reordering the variables.

\begin{proposition}
There exist constants $\rho > 0$ and $\varepsilon_0 > 0$ such that for all $\varepsilon \in (0, \varepsilon_0)$, the matrix $z_{ij} \pm z_i z_j$ is positive definite for all $x \in B_\rho(0)$. In particular, the constant $\rho$ is independent of $\varepsilon$.
\end{proposition}
\begin{proof}
At $\eta = 0$, by (\ref{eta=0 augmented}), it suffices to show the matrix $(z_{ij}+z_iz_j)_{1 \le i, j \le 2}$ is positive definite. Direct computation yields:
\begin{align*}
    z_{11}\pm z_1^2 &= 2(1+x_2^2)\varepsilon^\alpha \pm 4x_1^2(1+x_2^2)^2 \varepsilon^{2\alpha} \\ 
    &= \varepsilon^\alpha(2 + O(|x|^2))>0 \text{ if } |x| \text{ is sufficiently small}\\
    z_{12}\pm z_1 z_2 &= 4x_1 x_2 \varepsilon^\alpha \pm 4x_1 x_2 (1+x_1^2)(1+x_2^2) \varepsilon^{2\alpha}\\
    z_{22}\pm z_2^2 &= 2(1+x_1^2)\varepsilon^\alpha \pm 4x_2^2(1+x_1^2)^2 \varepsilon^{2\alpha} \\
    \det (z_{ij}+z_iz_j)_{1 \le i,j \le 2} &= \varepsilon^{2\alpha}(4 +O(|x|^2))>0 \text{ if } |x| \text{ is sufficiently small.}
\end{align*}

Now we consider the case when $\eta \neq 0$. First, we observe that the trailing $(n-2) \times (n-2)$ principal submatrix of $z_{ij}\pm z_i z_j$ (corresponding to indices $i, j \in \{3, \dots, n\}$), denoted by $M$, is given by (\ref{M}).

To show that $M$ is positive definite, we apply Sylvester's criterion. Let $M_i$ denote the $i$-th leading principal minor of $M$ and $I_i$ denote the $i \times i$ identity matrix for $i \in \{1, \dots, n-2\}$. By (\ref{fullu}), $\frac{u_3}{\eta}=\alpha(1+x_1^2)(1+x_2^2)r^{\alpha-2} >0$. Again, by (\ref{fullu}), we have: 
\begin{align*}
u_{33} \pm u_3^2 &= \alpha(\alpha-2)\eta^2r^{\alpha-4}+\alpha  r^{\alpha-2}+O(|x|^2 r^{\alpha-2}) \\
&=r^{\alpha-2}\left(\alpha(\alpha-2)\eta^2r^{-2}+\alpha+O(|x|^2 )\right)\\
&\geq r^{\alpha-2}\left(\alpha(\alpha-1)+O(|x|^2 )\right) >0 \text{  if $|x|$ is sufficiently small},
\end{align*}
where we used $\alpha<2$ and $\eta\leq r$ in the penultimate inequality and $\alpha>1$ in the last inequality. Hence, a computation similar to (\ref{detM}) gives:
\begin{align*}
    \det(M_i) &= \det\left( \frac{u_3}{\eta} I_i \right) \left[ 1 + \sum_{j=3}^{i+2} \left( \frac{u_{33} \pm u_3^2}{\eta^2} - \frac{u_3}{\eta^3} \right) x_j^2 \cdot \frac{\eta}{u_3} \right] \\
    &= \left( \frac{u_3}{\eta} \right)^{i}\left[ \frac{\sum_{j=3}^{n} x_j^2}{\sum_{j=3}^{i+2} x_j^2} + \sum_{j=3}^n \left( \frac{u_{33} \pm u_3^2}{\eta^2} - \frac{u_3}{\eta^3} \right) x_j^2 \cdot \frac{\eta}{u_3} \right] \frac{\sum_{j=3}^{i+2} x_j^2}{\sum_{j=3}^{n} x_j^2} \\
    &\geq  \left( \frac{u_3}{\eta} \right)^{i}\left[ 1 + \sum_{j=3}^n \left( \frac{u_{33} \pm u_3^2}{\eta^2} - \frac{u_3}{\eta^3} \right) x_j^2 \cdot \frac{\eta}{u_3} \right] \frac{\sum_{j=3}^{i+2} x_j^2}{\sum_{j=3}^{n} x_j^2}\\
    &= \left( \frac{u_3}{\eta} \right)^{i-1}(u_{33}\pm u_3^2)\frac{\sum_{j=3}^{i+2} x_j^2}{\sum_{j=3}^{n} x_j^2} > 0 \text{ if } |x| \text{ is sufficiently small}.
\end{align*}

We now show that the trailing principal $(n-1)\times(n-1)$ submatrix of $(z_{ij} \pm z_i z_j)$ (corresponding to indices $2 \le i,j \le n$) has a positive determinant. By an argument analogous to the reduction in Section \ref{step1}, we obtain the following modification of Pogorelov's formula:
\begin{equation*}
\det((z_{ij} \pm z_i z_j)_{2 \le i,j \le n}) = \left[ (u_{22} \pm u_2^2)(u_{33} \pm u_3^2) - (u_{23} \pm u_2 u_3)^2 \right] \left( \frac{u_3}{\eta} \right)^{n-3}.
\end{equation*}

By (\ref{fullu}), we have:
\begin{align*}
&\quad (u_{22}\pm u_2^2 )(u_{33}\pm u_3^2)-( u_{23}\pm u_2 u_3)^2 \\
&= \left[ 2r^\alpha + O(|x|^2)r^{2\alpha} \right] \left[ \alpha(\alpha-2)\eta^2 r^{\alpha-4} + \alpha r^{\alpha-2} + O(|x|^2) r^{\alpha-2} \right] \\
&\quad - \left[O(|x|)r^{\alpha-1}\right]^2 \\
& =r^{2\alpha-2}(2+O(|x|^2)r^\alpha)(\alpha+\alpha(\alpha-2)\eta^2r^{-2}+O(|x|^2)) \\
& \geq r^{2\alpha-2}(2+O(|x|^2)r^\alpha)(\alpha+\alpha(\alpha-2)+O(|x|^2)) \\
&= r^{2\alpha-2}(2\alpha(\alpha-1) + O(|x|^2)) > 0 \quad \text{if $|x|$ is sufficiently small}.
\end{align*}
Combining this with the fact that $\frac{u_3}{\eta} = \alpha(1+x_1^2)(1+x_2^2)(\varepsilon^2+\eta^2)^{\frac{\alpha}{2}-1} > 0$, we have shown that the submatrix $(z_{ij} \pm z_i z_j)_{2 \le i,j \le n}$ is positive definite.

Finally, we need to show that $z_{ij}\pm z_iz_j$ has a positive determinant. It suffices to show that the scaled matrix $\mathcal{M}$ defined in (\ref{anotherM}) has a positive determinant.
The entries of the matrix $\mathcal{M} = (m_{ij})$ are given explicitly by:

\begin{equation*}
    m_{i1} = \begin{bmatrix}
        2(1+x_2^2) \pm 4x_1^2(1+x_2^2)^2 r^\alpha \\
        4x_1x_2 \pm 4x_1x_2(1+x_1^2)(1+x_2^2)r^\alpha \\
        2\alpha x_1(1+x_2^2) \pm 2\alpha x_1 (1+x_1^2)(1+x_2^2)^2 r^{\alpha}
    \end{bmatrix}
\end{equation*}
\begin{equation*}
m_{i2} = \begin{bmatrix}
        4x_1x_2 \pm 4x_1x_2(1+x_1^2)(1+x_2^2)r^\alpha \\
        2(1+x_1^2) \pm 4x_2^2(1+x_1^2)^2 r^\alpha \\
        2\alpha x_2(1+x_1^2) \pm 2\alpha x_2 (1+x_2^2)(1+x_1^2)^2 r^{\alpha}
    \end{bmatrix}
\end{equation*}
\begin{equation*}
    m_{i3} =\begin{bmatrix}
        2\alpha x_1(1+x_2^2) \pm 2\alpha x_1 (1+x_1^2)(1+x_2^2)^2 r^{\alpha} \\
        2\alpha x_2(1+x_1^2) \pm 2\alpha x_2 (1+x_2^2)(1+x_1^2)^2 r^{\alpha} \\
        \alpha(\alpha-2)(1+x_1^2)(1+x_2^2) + \alpha(1+x_1^2)(1+x_2^2)\eta^{-2}r^2 \\
        \pm \alpha^2(1+x_1^2)^2(1+x_2^2)^2 r^{\alpha}
    \end{bmatrix}.
\end{equation*}
By inspecting the explicit entries of $\mathcal{M}$, we see that near the origin:
\begin{equation*}
\mathcal{M} = \begin{bmatrix}
2 + O(|x|^2) & O(|x|^2)  & O(|x|) \\
O(|x|^2)   & 2 + O(|x|^2)  & O(|x|) \\
O(|x|) & O(|x|) & m_{33}
\end{bmatrix}.
\end{equation*}

We estimate the determinant by expanding along the $(3,3)$ entry. First, we derive a lower bound for $m_{33}$ using the inequality $r^2 \geq \eta^2$:
\begin{align*}
m_{33} &= \alpha(\alpha-2)+ \alpha(1+O(|x|^2))\frac{r^2}{\eta^2} + O(|x|^2)+O(r^\alpha) \\
&\geq \alpha(\alpha-2) + \alpha +O(|x|^2 + r^\alpha) \\
&= \alpha(\alpha-1) + O(|x|^2) + O(r^\alpha).
\end{align*}
Since the leading principal $2 \times 2$ submatrix is a perturbation of $2I$, its determinant is positive for small $|x|$. The determinant of the full matrix is then:
\begin{align*}
\det \mathcal{M} \geq 4\alpha(\alpha-1) + O(|x|^2) +O(r^\alpha) > 0  \text{ if $|x|$ and $\varepsilon$ are sufficiently small},
\end{align*}
where we used $\alpha > 1$ in the last inequality.
\end{proof}
\section*{Acknowledgements}
The author would like to thank Yi Wang for her valuable guidance and supervision.
 \printbibliography
\end{document}